\title{Mather-Yau Theorem in Positive Characteristic}
\author{
       Gert-Martin Greuel and Thuy Huong Pham\\
 }
\date{17 April, 2014}

\documentclass[11pt]{article}
\usepackage{amsmath, amsthm, amssymb, mathrsfs, amscd}
\usepackage{indentfirst}

\newtheorem{theorem}{Theorem}[section]
\newtheorem{corollary}[theorem]{Corollary}
\newtheorem{proposition}[theorem]{Proposition }
\newtheorem{remark}[theorem]{Remark}

\newcommand{\Q}{\mathbb{Q}}
\newcommand{\N}{\mathbb{N}}
\newcommand{\C}{\mathbb{C}}

\begin{document}
\maketitle

\begin{abstract}
The well-known Mather-Yau theorem says that the isomorphism type of the local ring of an isolated complex hypersurface singularity is determined by its Tjurina algebra. It is also well known that this result is wrong as stated for power series $f$ in $K[[{\bf x}]]$ over fields $K$ of positive characteristic. In this note we show that, however, also in positive characteristic the isomorphism type of an isolated hypersurface singularity $f$ is determined by an Artinian algebra, namely by a "higher Tjurina algebra" for sufficiently high index, for which we give an effective bound. We prove also a similar version for the "higher Milnor algebra" considered as $K[[f]]$-algebra.
\end{abstract}
\maketitle

\section{Introduction}
Let $K$ be an algebraically closed field of arbitrary characteristic and $K[[{\bf{x}}]]=K[[x_1, x_2,..., x_n]]$ the formal power series ring over $K$ with maximal ideal $\mathfrak{m}$. Let $f\in K[[{\bf{x}}]]$. We denote by 
\vskip 6pt \hskip 50pt 
${\rm {j}}(f)=\Big\langle\frac{\partial f}{\partial{x_1}},\frac{\partial f}{\partial{x_2}},...,\frac{\partial f}{\partial{x_n}}\Big\rangle$, the Jacobian ideal of $f$,
\vskip 6pt \hskip 50pt 
$M(f)=K[[{\bf{x}}]]/{\rm {j}}(f)$,  the Milnor algebra of $f$, and
\vskip 6pt \hskip 50pt 
$\mu(f)=dim_KM(f)$,  the Milnor number of $f$.
\vskip 6pt\noindent   Moreover, we call
\vskip 6pt \hskip 50pt 
${\rm {tj}}(f)=\left\langle f, \frac{\partial f}{\partial{x_1}},\frac{\partial f}{\partial{x_2}},...,\frac{\partial f}{\partial{x_n}}\right\rangle$, the Tjurina ideal of $f$,
\vskip 6pt \hskip 50pt  
$T(f)=K[[{\bf{x}}]]/{\rm tj}(f)$, the Tjurina algebra of $f$, and 
\vskip 6pt \hskip 50pt 
$\tau(f)=dim_KT(f)$, the Tjurina number of $f$.
\vskip 6pt\noindent  More generally, for  $k\in \N$, set
\vskip 6pt \hskip 90pt $T_k(f)=K[[{\bf{x}}]] \big/\big\langle f, \mathfrak{m}^k{\rm {j}}(f)\big\rangle$ resp. 
\vskip 6pt \hskip 90pt$M_k(f)=K[[{\bf{x}}]]/ \mathfrak{m}^k{\rm {j}}(f),$
\vskip 6pt \noindent and call it the \textit{$k$-th Tjurina} resp. \textit{$k$-th Milnor algebra} of $f$.

\vskip 5pt Two power series $f$ and $g$ in $K[[{\bf{x}}]]$ are said to be \textit{right equivalent}, denoted $f\mathop  \sim \limits^r g$, if there is an automorphism $\varphi\in Aut(K[[{\bf{x}}]])$ such that $g=\varphi(f)$. They are called \textit {contact equivalent}, denoted $f\mathop  \sim \limits^c g$, if there are $\varphi\in Aut(K[[{\bf{x}}]])$ and a unit $u\in K[[{\bf{x}}]]^\ast$ such that $g=u\varphi(f)$. If $f\mathop  \sim \limits^r g$ then the associated Milnor algebras are isomorphic and if  $f\mathop  \sim \limits^c g$ then the associated Tjurina algebras are isomorphic. 
The theorem of Mather and Yau says:

\begin{theorem}[{\cite[Theorem 1]{MY82}}] \label{MYoriginal}
Let $f, g\in \mathfrak{m}\subset \C\{\bf{x}\}$ be such that $\tau(f)<\infty$. The following are equivalent:\\
i)  $f\mathop  \sim \limits^c g$.\\
ii) $T(f)\cong T(g)$ as $\C$-algebras.
\end{theorem}

The theorem was slightly generalized in \cite[Theorem 2.26]{GLS07} (without assuming isolated singularity):

\begin{theorem}\label{GLScontact} 
Let $f, g\in \mathfrak{m}\subset \C\{\bf{x}\}$. The following are equivalent:\\
i) $f\mathop  \sim \limits^c g$.\\
ii)  For all $k\ge 0$, $T_k(f)\cong T_k(g)$ as $\C$-algebras.\\
iii) There is some $k\ge 0$ such that  $T_k(f)\cong T_k(g)$ as $\C$-algebras.\\
In particular, $f\mathop  \sim \limits^c g$ iff $T(f)\cong T(g)$ as  $\C$-algebras.
\end{theorem}

However, the theorem is not true if $K$ has characteristic $p>0$ in general, as was already noted by Mather and Yau \cite{MY81}. For $f=x^{p+1} + y^{p+1}$ and $g=f+x^p$ we have $f\mathop {\not  \sim }\limits^c g$ but $T(f)=T(g)$.  

It is also known by \cite[Theorem 2]{Yau84} that over the complex numbers the Milnor algebra $M(f)$ determines $f$ up to right equivalence, if we consider $M(f)$ as $\C\{t\}$-algebra where $t$ acts by multiplication with $f$ (but not as $\C$-algebra). 
The following generalization can be deduced from Theorem {\ref{GLScontact}} (cf. \cite[Theorem 2.28]{GLS07}):

\begin{theorem}\label{GLSright}
Let $f, g\in \mathfrak{m}\subset \C\{{\bf{x}}\}$ be hypersurface singularities. Then the following are equivalent:\\
i) $f\mathop  \sim \limits^r g$.\\
ii) For all $k\ge 0$, $M_k(f)\cong M_k(g)$ as $\C\{t\}$-algebras.\\
iii) There is some $k\ge 0$ such that  $M_k(f)\cong M_k(g)$ as $\C\{t\}$-algebras.\\
In particular, $f\mathop  \sim \limits^r g$ iff $M(f)\cong M(g)$ as $\C\{t\}$-algebras.
\end{theorem}

The same $f$ and $g$ as in the example above are not right equivalent but we have ${\rm {j}}(f)={\rm {j}}(g)$ and hence $M_k(f)=M_k(g)$ for all $k$ as $K$-algebras. Moreover, considered as $K[[t]]$-algebras, $M(f)\cong M(g)$.

Our aim is to see how far the Mather-Yau theorem and Theorem {\ref{GLSright}} hold in the case of positive characteristic. In order to do that we need some additional notions.

For $k\in\N$ we say that $f$ is \textit{right} (respectively \textit {contact}) $k$\textit {-determined} if it is right (respectively contact) equivalent to every $g\in K[[{\bf{x}}]]$ satisfying $g-f\in\mathfrak{m}^{k+1}$. We denote by $ord(f)$ the order (or multiplicity) of $f$, i.e. the maximal $l$ such that $f\in \mathfrak{m}^l$, which is invariant under right and contact equivalence.

Note that the proof given by Mather and Yau (as well as in \cite{GLS07}) uses integration of vector fields and cannot be generalized to positive characteristic. To prove an appropriate generalization of Theorem {\ref{MYoriginal}} - {\ref{GLSright}} in positive characteristic we need the finite determinacy theorem proved in \cite{BGM12}.

\begin{theorem}[{\cite[Theorem 3]{BGM12}}]\label{BGM12}
Let $0\ne f\in \mathfrak{m}^2$ and $k\in \N$.\\
i) If $\mathfrak{m}^{k+2}\subset \mathfrak{m}^2{\rm {j}}(f)$ then $f$ is right $(2k-ord(f)+2)$-determined.\\
ii) If $\mathfrak{m}^{k+2}\subset \mathfrak{m}\left\langle f\right\rangle +\mathfrak{m}^2{\rm {j}}(f)$ then $f$ is contact $(2k-ord(f)+2)$-determined.
\end{theorem}

\section{Results}\label{results}
Let us first consider the case of an arbitrary algebraically closed field $K$ of characteristic 0.

\begin{proposition}
Let $K$ be an algebraically closed field of characteristic 0 and $f, g \in\mathfrak{m}\subset K[[{\bf{x}}]]$ with $f$ an isolated singularity. 
Then the statements of Theorem {\ref{GLScontact}} (resp. Theorem {\ref{GLSright}}) hold 
if we replace $\C$ by $K$, $\C\{{\bf{x}}\}$ by $K[[{\bf{x}}]]$, and $\C\{t\}$ by $K[[t]]$.
\end{proposition}

\begin{proof} Note that (in characteristic 0) $\tau(f)<\infty$ iff $\mu(f)<\infty$ and this is equivalent to $f$ having an isolated singularity. Moreover, since the difference between $\tau(f)$ and $dim_KT_k(f)$ (resp. $\mu(f)$ and $dim_KM_k(f)$) is finite, it follows that any of the equivalent statements of Theorem {\ref{GLScontact}} (resp. Theorem {\ref{GLSright}}) implies that $g$ has also an isolated singularity. 

The proof uses a "Lefschetz principle", that is we prove it through the complex case. We give the proof only for the interesting direction that iii) implies i). Moreover, we consider only the case of contact equivalence, the proof for right equivalence works along the same lines. 

Assume that $f, g\in \mathfrak{m}\subset K[[{\bf{x}}]]$ be such that the finite dimensional $K$-algebras $T_k(f)$ and $T_k(g)$ are isomorphic for some $k\ge 0$. By the Lifting lemma \cite[Lemma 1.23]{GLS07}, this isomorphism lifts to a $K$-algebra automorphism $\theta: K[[{\bf x}]]\to K[[{\bf x}]]$ such that $\theta \big(\big\langle f, \mathfrak{m}^k{\rm j}(f)\big\rangle\big)=\big\langle g, \mathfrak{m}^k{\rm j}(g)\big\rangle$.  
 
Let $l_0,..., l_r$ be the generators $f, m_j\cdot \frac{\partial f}{\partial x_i}$ of the ideal $\big\langle f, \mathfrak{m}^k{\rm j}(f)\big\rangle$, with $m_j$ generators of $\mathfrak{m}^k$, and $w_0,..., w_r$ analogous generators of $\big\langle g, \mathfrak{m}^k{\rm j}(g)\big\rangle$. Then for $i=0,..., r$, we have 
$$\theta(l_i)=b_{0,i}w_0+...+b_{r,i}w_r, \hskip 10pt \theta^{-1}(w_i)=d_{0,i}l_0+...+d_{r,i}l_r$$
for some $b_{j,i}$, $d_{j,i}\in K[[{\bf x}]]$, $j=0,..., r$.

Let $A\subset K$ be the set consisting of all coefficients of $l_i$ and $w_i$, $i=0,..., r$, the coefficients of $\theta(x_t)$, $t=1,...,n$, and the coefficients of $b_{j,i}$ and $d_{j,i}$, $j, i=0,..., r$. Let $K':=\Q (A)\subset K$ be the subfield generated by $A$. Then $l_i$, $w_i$, $i=0,..., r$, are in $K'[[{\bf x}]]$ and $\theta$ induces a $K'$-algebra automorphism $\theta: K'[[{\bf{x}}]]\to K'[[{\bf{x}}]]$. If $I=\langle l_0,..., l_r\rangle \cdot K'[[{\bf x}]]$ and $J=\langle w_0,..., w_r\rangle \cdot K'[[{\bf x}]]$ denote the ideals in $K'[[{\bf x}]]$, then $\theta(I)=J$. Hence, we obtain a $K'$-algebra isomorphism $K'[[{\bf{x}}]]/I\cong K'[[{\bf{x}}]]/J.$

The countable extension $K'$ of $\Q $ is $\Q $-isomorphic to a subfield $\tilde K$ of $\C$, hence there is an isomorphism $\varphi: K'[[{\bf{x}}]] \to {\tilde K}[[{\bf{x}}]] $ such that $\varphi |_{\Q[[{\bf{x}}]]}=id_{\Q[[{\bf{x}}]]}$. Let $\tilde f=\varphi(f)$ and $\tilde g=\varphi(g)$. Furthermore, set $\tilde I:=\big\langle \tilde f, \mathfrak{m}^k{\rm j}(\tilde f)\big\rangle\cdot \tilde K[[{\bf{x}}]]$ and $\tilde J:=\big\langle \tilde g, \mathfrak{m}^k{\rm j}(\tilde g)\big\rangle\cdot \tilde K[[{\bf{x}}]]$. Then $\varphi(I)=\tilde I$ and $\varphi(J)=\tilde J$, and thus we get a $\tilde K$-algebra isomorphism $\tilde K[[{\bf{x}}]]/\tilde I\cong \tilde K[[{\bf{x}}]]/\tilde J.$ By taking the complete tensor product with $\C$ over the field $\tilde K$, this $\tilde K$-algebra  isomorphism extends to $\C$ so that $T_k(\tilde f)\cong T_k(\tilde g)$ as $\C$-algebras. Since both algebras are finite dimensional, $\tilde f$ and $\tilde g$ have isolated singularities and are therefore finitely determined. 

Hence $\tilde f$ and $\tilde g$ are formally contact equivalent to polynomials, in particular to convergent power series, and we can apply Theorem {\ref{GLScontact}}, providing convergent contact equivalence. Altogether, there are $\psi\in Aut(\C[[{\bf{x}}]])$ and a unit $u$, with  \vskip 8pt \hskip60pt $\psi_t=\psi(x_t)=\sum\limits_{\gamma}c_{\gamma}^{(t)}{\bf{x}}^{\gamma}$, \hskip 5pt and \hskip 3pt $u=\sum\limits_{\delta}u_\delta{\bf{x}}^{\delta}\in \C[[{\bf{x}}]]^{\ast}$
 \vskip 3pt \noindent such that 
$$(\star) \hskip 30pt  \tilde g = u\psi(\tilde f).$$  

We first claim that there are $u\in\bar{\tilde {K}}[[{\bf{x}}]]^{\ast}$ and $\psi\in Aut(\bar{\tilde {K}}[[{\bf{x}}]])$, with $\bar{\tilde {K}}$ the algebraic closure of $\tilde K$, satisfying ($\star$). By comparing the coefficients in ($\star$) (with $\{c_{\gamma}^{(t)}, u_{\delta}\}$ indeterminates) up to degree $q$ we get an ideal $I_q$ generated by polynomials in $\tilde K[B_q]$ where $B_q$ is a finite subset of size, say $N_q$, of $\{c_{\gamma}^{(t)}, u_{\delta}\}$. Because of ($\star$) the zero set $Z( I_q) \ne \emptyset$ in $\C^{N_q}$ and $Z(I_{q+1})$ extends $Z(I_q)$. By Hilbert's Nullstellensatz $Z(I_q)$ is already contained in $\bar{\tilde K}^{N_q}$. This proves the claim by induction on $q$.

Now let $u$ and $\psi=(\psi_1,...,\psi_n)$ as above, with $u_{\delta}$ and $c_{\gamma}^{(t)}$ algebraic over $\tilde K$. Hence, there are univariate polynomials $\tilde h_{\gamma}^{(t)}, \tilde h_{\delta}\in {\tilde K}[z]$ such that $\tilde h_{\gamma}^{(t)}(c_{\gamma}^{(t)})=0$ and $\tilde h_{\delta}(u_\delta)=0$. Let $\tilde \varphi: K'[z] \to \tilde K[z]$ be the isomorphism induced by the isomorphism $ K' \cong \tilde K$. Now we define $h_\gamma^{(t)}:=\tilde \varphi^{-1}(\tilde h_\gamma^{(t)})$ and $h_\delta:=\tilde \varphi^{-1}(\tilde h_{\delta})$ in $K'[z]$. 

Choose $v_\delta, a_\gamma^{(t)}\in \bar K'\subset K$ to be roots of $h_\delta$ $h_\gamma^{(t)}$, and define
$$v:=\sum\limits_{\delta}v_\delta{\bf{x}}^\delta, \hskip 10pt \phi_t:=\sum\limits_{\gamma}a_\gamma^{(t)}{\bf{x}}^{\gamma}.$$
Then $v\in \bar{K'}[[{\bf{x}}]]^{\ast}\subset K[[{\bf{x}}]]^{\ast}$ and $\phi=(\phi_1,...,\phi_n)\in Aut (K[[{\bf{x}}]])$. We now show that $g=v\phi(f)$.
Let, by abuse of notation, $B_q\subset \{c_\gamma^{(t)}, u_\delta\}\subset{\bar{\tilde K}}$ be the set of all components of the solutions $Z(I_q)\subset ({{\bar{\tilde K}})^{N_q}}$. Let $A_q\subset \{a_\gamma^{(t)}, v_\delta\}\subset \bar {K'}$ be the set corresponding to $B_q$. We have towers of finite separable field extensions
$$\tilde K\subset \tilde K(B_q)\subset \tilde K(B_{q+1})\subset...\subset\C$$ 
and 
$$K'\subset  K'(A_q)\subset  K'(A_{q+1})\subset...\subset K $$
and, by construction, compatible isomorphisms
$K'(A_q)\cong\tilde K(B_q) $ over $ K'\cong\tilde K$. These induce an isomorphism over $K' \cong \tilde K$
$$L':=\bigcup\limits_{q\ge 0}K'(A_q) \to \tilde L:=\bigcup\limits_{q\ge 0}\tilde K(B_q) $$
of subfields of $K$ resp. $\C$ which induces an isomorphism
$$\hat {\varphi}: L'[[{\bf{x}}]] \to \tilde L[[{\bf{x}}]],$$
extending $\varphi : K'[[{\bf{x}}]] \to \tilde K[[{\bf{x}}]]$.
Then we have $\hat{\varphi}(f)=\tilde f$, $\hat\varphi(g)=\tilde g$, $\hat\varphi(v)= u$, and $\hat\varphi (\phi_t)=\psi_t$. This implies
$$\hat\varphi(g)=\tilde g= u\psi(\tilde f)=\hat\varphi (v)\psi(\hat\varphi(f))=\hat\varphi (v)\hat\varphi(\phi(f))=\hat\varphi(v\phi(f)),$$
hence $g=v\phi(f).$
\end{proof}

Now we formulate our main results for $K$ an algebraically closed field of any characteristic.

\begin{theorem}{\label{contact}}
Let $f,g\in K \left[\left[\bf{x} 
 \right]\right]$ be such that $ord(f)=s \ge 2$ and $\tau(f)<\infty$. 
Then the following are equivalent:\\
i) $f\mathop  \sim \limits^c g$. \\
ii) $T_k(f)  \cong T_k(g)$ as $K$-algebras for some (equivalently for all) $k$ such that $$\mathfrak{m}^{\left\lfloor {\frac{{k + 2s}}{2}} \right\rfloor}\subset\mathfrak{m}\left\langle f \right\rangle +\mathfrak{m}^2{\rm {j}}(f)$$
\noindent where ${\left\lfloor {\frac{{k + 2s}}{2}} \right\rfloor}$ means the maximal integer which does not exceed $\frac{{k + 2s}}{2}$.
\end{theorem}

\begin{corollary}\label{contact1}
Let $f$ and $g$ be as in Theorem {\ref{contact}}. Then
 $f\mathop  \sim \limits^c g$ iff
 $T_k(f)  \cong T_k(g)$ as $K$-algebras for some (equivalently for all) $k\ge2\tau(f)-2s+4$.
\end{corollary}

 We set for any ideal $I\subset K[[{\bf{x}}]]$,
$$ord(I):=max\{l\in {\N}\hskip3pt| \hskip 3pt I\subset \mathfrak{m}^l\}.$$

\begin{theorem}{\label{right}}
Let $f,g\in K \left[\left[\bf{x} 
 \right]\right]$ be such that $ord(f)=s \ge 2$ and $\mu(f)<\infty$. Let $s'=ord({\rm {j}}(f))$. Then the following are equivalent:\\
i) $f\mathop  \sim \limits^r g$. \\
ii) $M_k(f)  \cong M_k(g)$ as $K[[t]]$- algebras for some (equivalently for all) $k$ such that $$\mathfrak{m}^{\left\lfloor {\frac{{k + s+s'+1}}{2}} \right\rfloor}\subset\mathfrak{m}^2{\rm {j}}(f).$$
\end{theorem}

\begin{corollary}\label{right1}
Let $f$ and $g$ be as in Theorem {\ref{right}}. Then
$f\mathop  \sim \limits^r g$ iff
 $M_k(f)  \cong M_k(g)$ as $K[[t]]$-algebras for some (equivalently for all) $k\ge2\mu(f)-s-s'+3$.
\end{corollary}

\begin{remark}
{\rm (1) Condition i) of Theorem {\ref{contact}} (resp. Theorem {\ref{right}}) implies $T_k(f)  \cong T_k(g)$ (resp. $M_k(f)  \cong M_k(g)$) for all $k$ without assuming $\tau$ (resp. $\mu$) to be finite, see the proofs of $i) \Rightarrow ii)$.
On the other hand, the inclusion relation in Theorem {\ref{contact}} ii) (resp. Theorem {\ref{right}} ii)) implies already, that $\tau(f)$ (resp. $\mu(f)$) is finite by Theorem {\ref{BGM12}} and \cite[Theorem 4]{BGM12}.\\
\noindent(2) Since ideal-membership in power series rings can be effectively tested (e.g. by standard basis methods, cf. \cite{GP08} and \cite{DGPS12}) the bounds for $k$ in Theorems {\ref{contact}} and {\ref{right}} can be effectively computed. Corollaries {\ref{contact1}} resp. {\ref{right1}} provide the simple bounds $k\ge2\tau(f)$ resp. $k\ge2\mu(f)$.\\
\noindent (3) The inclusion relation in ii) implies $k\ge 2$ in Theorem {\ref{contact}}, which is necessary as the following example shows. Take $f=y^2+x^3y$ and $g=f+x^5$ in characteristic $2$, then $T_k(f)=T_k(g)$ for $k=0,1$ but $f\mathop{\not\sim}\limits^{c} g$ ($f$ has two branches and $g$ is irreducible as can be verified by using Singular \cite{DGPS12}).\\
\noindent(4) It was proved in \cite{Sho76}, see also \cite{Yau83}, that for an isolated quasi-homogeneous singularity $f\in \mathfrak{m}\subset \C\{{\bf x}\}$ and any $g\in \mathfrak{m}$, $M(f)\cong M(g)$ as $\C$-algebras implies that $f\mathop  \sim \limits^r g$. This theorem does not have an analogue in characteristic $p$, even if we use the higher Milnor algebras. For example, for the homogeneous polynomials $f=x^{p+1}+y^{p+1}$ and $g=f+x^p$, we have $M_k(f)=M_k(g)$ as $K$-algebras for all $k$, but $f$ is not right equivalent to $g$.}
\end{remark}

\section{Proofs}

\begin{proof}[Proof of Theorem {\ref{contact}}]
$i) \Rightarrow ii).$ By definition of contact equivalence, there are $\varphi  \in$ $Aut(K[[{\textbf{x}}]])$ and $u \in K{[[ \textbf{x}]]^ \ast }$ such that $g=u\varphi (f)$. For arbitrary $k$, we have
\begin{align*}
 \left\langle g, \mathfrak{m}^k{\rm {j}}(g)\right\rangle = &\left\langle {u\varphi (f),{\mathfrak{m}^k}{\rm {j}}\big({u\varphi (f)} \big)} \right\rangle  = \left\langle {\varphi (f),{\mathfrak{m}^k}{\rm {j}}\big(\varphi (f)\big)} \right\rangle \\
=  &\left\langle {\varphi (f),{\mathfrak{m}^k}\varphi \big({\rm {j}}(f)\big)} \right\rangle  = \varphi\big(\big\langle {f,{\mathfrak{m}^k}{\rm {j}}(f)} \big\rangle\big).
\end{align*}
$ii) \Rightarrow i).$ Suppose that for some $k$ such that $\mathfrak{m}^{\left\lfloor {\frac{{k + 2s}}{2}} \right\rfloor}\subset\mathfrak{m}\left\langle f \right\rangle +\mathfrak{m}^2{\rm {j}}(f),$
 $\varphi $ is an isomorphism of the $K$-algebras in ii). Then by the Lifting lemma \cite[Lemma 1.23]{GLS07}, which works for formal power series over any field with the trivial valuation, $ \varphi$ lifts to an isomorphism $\tilde \varphi $: $K[[{\bf x}]]\rightarrow K[[{\bf x}]]$ with $\tilde \varphi\big(\big\langle {f,{\mathfrak{m}^k}{\rm {j}}(f)} \big\rangle\big)=\big\langle {g,{\mathfrak{m}^k}{\rm {j}}(g)} \big\rangle$. Since
$$\tilde \varphi \left( {\left\langle {f,{\mathfrak{m}^k}{\rm {j}}(f)} \right\rangle } \right) = \left\langle {\tilde \varphi (f),{\mathfrak{m}^k}\tilde \varphi ({\rm {j}}(f))} \right\rangle  = \left\langle {\tilde\varphi (f),{\mathfrak{m}^k}{\rm {j}}(\tilde\varphi (f))} \right\rangle,$$
we may assume that
$$\left\langle {f,{\mathfrak{m}^k}{\rm {j}}(f)} \right\rangle  = \left\langle {g,{\mathfrak{m}^k}{\rm {j}}(g)} \right\rangle. $$

This implies $f=h_1g+H$ for some $h_1\in K[[{\bf x}]]$ and $H\in \mathfrak{m}^k{\rm {j}}(g)$. Since $f\in \mathfrak{m}^s$, ${\rm {j}}(f)\subset\mathfrak{m}^{s-1}$ and hence $\mathfrak{m}^k{\rm {j}}(f)\subset\mathfrak{ m}^{k+s-1}\subset \mathfrak{m}^s$ so that $g\in \mathfrak{m}^s$ and then $H\in \mathfrak{m}^{k+s-1}$. Consider two cases:
\vskip4pt\noindent Case 1: $h_1$ is a unit. Since $\mathfrak{m}^{\left\lfloor {\frac{{k + 2s}}{2}} \right\rfloor}\subset\mathfrak{m}\left\langle f \right\rangle +\mathfrak{m}^2{\rm {j}}(f),$ by Theorem {\ref{BGM12}}, $f$ is contact \big($2\big\lfloor {\frac{{k + 2s}}{2}} \big\rfloor- s -2$\big)-determined. Since $f-h_1g\in \mathfrak{m}^{k+s-1}$ and $k+s-1>2\frac{k+2s}{2}-s-2\ge2\left\lfloor {\frac{{k + 2s}}{2}} \right\rfloor-s-2$, we have ${h_1}g\mathop  \sim \limits^c f$. Moreover, since $h_1$ is a unit, $g\mathop  \sim \limits^c {h_1}g$. Hence, $g\mathop  \sim \limits^c f$.\\
Case 2: $h_1$ is not a unit. Since $g=h_2f+G$ for some  $h_2\in K[[{\bf x}]]$ and $G\in \mathfrak{m}^k{\rm {j}}(f)$, we have
$$f=h_1g+H=h_1(h_2f+G)+H=h_1h_2f+h_1G+H.$$
Since $h_1\in\mathfrak{ m}$ and $G\in \mathfrak{m}^{k+s-1}$, this implies
$$(1-h_1h_2)f=h_1G+H\in \mathfrak{m}^{k+s-1}.$$
Hence, $f\in\mathfrak{ m}^{k+s-1}$ since $1-h_1h_2$ is a unit. On the other hand since $ord(f)=s$, $f\not  \in\mathfrak{ m}^{s+1}$. Therefore, $k\le1$, a contradiction.
\end{proof}

\begin{proof}[Proof of Corollary {\ref{contact1}}]
If  $f\mathop  \sim \limits^c g$ then the two $K$-algebras $T_k(f)$  and $T_k(g)$ are isomorphic for all $k$ by  Theorem {\ref{contact}}. Conversely, suppose that $T_k(f)\cong T_k(g)$ for $k\ge2\tau(f)-2s+4$. The fact that $dim_KT(f)=\tau$ implies ${\mathfrak{m}}^\tau\subset {\rm tj}(f)$ so that
$$\mathfrak{m}^{\tau+2}\subset\mathfrak{m}^2\left\langle f\right\rangle+ \mathfrak{m}^2{\rm {j}}(f)\subset\mathfrak{m}\left\langle f \right\rangle +\mathfrak{m}^2{\rm {j}}(f).$$
On the other hand, since $k\ge 2\tau-2s+4$ we have $\frac{k+2s}{2}\ge \tau+2$ so that $\left\lfloor {\frac{{k + 2s}}{2}} \right\rfloor\ge \tau+2$. This implies  $\mathfrak{m}^{\left\lfloor {\frac{{k + 2s}}{2}} \right\rfloor}\subset\mathfrak{m}^{\tau+2} \subset\mathfrak{m}\left\langle f \right\rangle +\mathfrak{m}^2{\rm {j}}(f)$. By Theorem {\ref{contact}}, we get $f\mathop  \sim \limits^c g$.
\end{proof}

\begin{proof}[Proof of Theorem {\ref{right}}]
$i)\Rightarrow ii).$ By assumption, there is a $\phi\in Aut(K[[{\bf{x}}]])$ such that $g=\phi(f)$. Then for arbitrary $k$, we have
$$\phi(\mathfrak{m}^k{\rm {j}}(f))=\mathfrak{m}^k{\rm {j}}(\phi(f))=\mathfrak{m}^k{\rm {j}}(g).$$
The $K$-algebra automorphism $\phi$ induces the $K$-algebra isomorphism 
$$\bar\phi: M_k(f)\to M_k(g), \hskip 10pt \bar{h}\mapsto \overline{\phi(h)},$$
which is a $K[[t]]$-algebra isomorphism since $\bar\phi(\bar f)=\overline{\phi(f)}=\bar g$.

\noindent $ii)\Rightarrow i).$ Using the Lifting lemma, the $K[[t]]$-algebra isomorphism $$\bar\phi : M_k(f) \to M_k(g)$$ 
lifts to a $K$-algebra isomorphism $\phi \in Aut(K[[{\bf{x}}]])$ such that $\phi(\mathfrak{m}^k{\rm {j}}(f))=\mathfrak{m}^k{\rm {j}}(g).$ Since $\bar\phi(\bar f)=\bar g$ and $\bar\phi(\bar f)=\overline {\phi(f)}$ in $M_k(g)$, we get
$\phi(f)-g\in \mathfrak{m}^k{\rm {j}}(g)=\mathfrak{m}^k{\rm {j}}(\phi(f))$. Since $ord(f)$, $ord({\rm {j}}(f))$, and the degree of the right determinacy of $f$ are invariant under right equivalence, we reduce to the situation $f-g\in \mathfrak{m}^k{\rm {j}}(f)$. This implies $f-g\in \mathfrak{m}^{k+s'}$. Moreover, since $\mathfrak{m}^{\left\lfloor {\frac{{k + s+s'+1}}{2}} \right\rfloor}\subset\mathfrak{m}^2{\rm {j}}(f)$,  by Theorem {\ref{BGM12}}, $f$ is right $(2{\left\lfloor {\frac{{k + s+s'+1}}{2}} \right\rfloor}-s-2)$-determined. Since $k+s'>2{\left\lfloor {\frac{{k + s+s'+1}}{2}} \right\rfloor}-s-2$, we get $f\mathop  \sim \limits^r g$.
\end{proof}

\begin{proof}[Proof of Corollary {\ref{right1}}]
We proved in Theorem {\ref{right}} that if  $f\mathop  \sim \limits^r g$ then  $M_k(f)$  and $M_k(g)$ are isomorphic as $K[[t]]$-algebras for all $k$. Conversely, since $dim_KM(f)=\mu$, $\mathfrak{m}^{\mu+2}\subset\mathfrak{m}^2{\rm {j}}(f).$ Since $k\ge 2\mu-s-s'+3$, we have $\left\lfloor {\frac{{k + s+s'+1}}{2}} \right\rfloor\ge\mu+2$ so that
$$\mathfrak{m}^{\left\lfloor {\frac{{k + s+s'+1}}{2}} \right\rfloor}\subset\mathfrak{m}^{\mu+2} \subset\mathfrak{m}^2{\rm {j}}(f).$$
By Theorem {\ref{right}}, we get $f\mathop  \sim \limits^r g$.
\end{proof}

{\textit {Acknowledgments}}:
The second author was supported by DAAD (Germany). We like to thank the referees for careful proofreading.


{\scshape Fachbereich Mathematik, Universit\"{a}t Kaiserslautern,\\
\indent Erwin-Schr\"{o}dinger Stra$\ss$e,
67663 Kaiserslautern,  Germany.}\\
\indent{\it E-mail address}: {greuel@mathematik.uni-kl.de}

{\scshape Fachbereich Mathematik, Universit\"{a}t Kaiserslautern,\\
\indent Erwin-Schr\"{o}dinger Stra$\ss$e,
67663 Kaiserslautern,  Germany.}\\
\indent{\it E-mail address}: {hpham@mathematik.uni-kl.de}


\begin{thebibliography}{9}

\bibitem [1]{BGM12} Y. Boubakri, G.-M. Greuel, and T. Markwig, \textit{Invariants of hypersurface singularities in positive characteristic}, Rev. Mat. Complut. {\bf 25} (2012), no. 1, 61-85. 

\bibitem [2]{DGPS12} W. Decker, G.-M. Greuel, G. Pfister, and  H. Sch\"{o}nemann, \textit{Singular 3-1-5 - A computer algebra system for polynomial computations}, Center for Computer Algebra, University of Kaiserslautern, 2012, http://www.singular.uni-kl.de.

\bibitem [3]{GLS07} G.-M. Greuel, C. Lossen, and E. Shustin, \textit{Introduction to singularities and deformations}, Springer, Berlin, 2007.

\bibitem [4]{GP08} G.-M. Greuel and G. Pfister, \textit{A Singular introduction to commutative algebra}, 2nd ed.,  Springer, Berlin, 2008. 

\bibitem [5]{MY81} J. N. Mather and S. S.-T. Yau, \textit{Criterion for biholomorphic equivalence of isolated hypersurface singularities}, Proc. Natl. Acad. Sci. USA {\bf 78}  (1981), no. 10, 5946-5947. 

\bibitem [6]{MY82} J. N. Mather and S. S.-T. Yau, \textit{Classification of isolated hypersurface singularities by their moduli algebras}, Invent. Math. {\bf 69} (1982), no. 2, 243-251. 

\bibitem [7]{Sho76} A. N. Shoshitaishvili, \textit{Functions with isomorphic Jacobian ideals}, Funct.  Anal. Appl. {\bf 10} (1976), no. 2, 128-133. 

\bibitem [8]{Yau83} S. S.-T. Yau, \textit{Milnor algebras and equivalence relations among holomorphic functions}, Bull. Amer. Math. Soc. (N.S.) {\bf 9} (1983), no. 2, 235-239. 

\bibitem [9]{Yau84} S. S.-T. Yau, \textit{Criteria for right-left equivalence and right equivalence of
holomorphic functions with isolated critical points}, Proc.
Sympos. Pure Math. {\bf 41} (1984), 291-297. 

\end{thebibliography}
\end{document}